\documentclass[a4paper]{article}
\usepackage{amsmath, amsthm, amssymb, enumerate, thmtools}

\usepackage{tikz}
\usetikzlibrary{shapes, decorations.pathreplacing}

\newcommand{\abs}[1]{\lvert#1\rvert}
\newtheorem{thm}{Theorem}
\newtheorem{lem}[thm]{Lemma}
\theoremstyle{remark}
\newtheorem*{rem}{Remark}
\newtheorem{obs}{Observation}[thm]
\newtheorem{clm}{Claim}[thm]
\newtheorem{case}{Case}
\declaretheorem[name=Proof of Claim,style=remark,qed=$\lozenge$,numbered=no]{poc}
\newcommand{\thrd}{\cdots}
\newcommand{\rlp}{\operatorname{rl_{p}}}
\newcommand{\rls}{\operatorname{rl_{s}}}
\newcommand{\gm}[1][m]{G^{1/#1}}

\newcommand*{\bfrac}[2]{\genfrac{(}{)}{}{}{#1}{#2}}
\newcommand*{\theorem}[1]{Theorem~\ref{#1}}
\newcommand*{\lemma}[1]{Lemma~\ref{#1}}
\newcommand*{\observ}[1]{Observation~\ref{#1}}
\newcommand{\ie}{i.e.\ }
\newcommand*{\floor}[1]{\lfloor#1\rfloor}
\newcommand*{\floorf}[2]{\genfrac{\lfloor}{\rfloor}{}{}{#1}{#2}}
\newcommand*{\ceil}[1]{\lceil#1\rceil}
\newcommand*{\ceilf}[2]{\genfrac{\lceil}{\rceil}{}{}{#1}{#2}}
\newcommand*{\seq}[3][1]{#2_{#1},\ldots,#2_{#3}}
\newcommand*{\level}[2]{[#1]^{(#2)}}

\makeatletter
\let\@fnsymbol\@arabic
\makeatother

\newcommand\thankx[1]{\begingroup\let\rlap\relax\thanks{#1}\endgroup}

\begin{document}

\title{Locating a robber with multiple probes}
\author{John Haslegrave\thankx{University of Warwick, Coventry, UK. {\tt j.haslegrave@cantab.net}}, Richard A. B. Johnson\thankx{The King's School, Canterbury, UK. {\tt rabj@kings-school.co.uk}}, Sebastian Koch\thankx{University of Cambridge, Cambridge, UK.}}
\maketitle

\begin{abstract}
We consider a game in which a cop searches for a moving robber on a connected graph using distance probes, which is a slight variation on one introduced by Seager (Seager, 2012). Carragher, Choi, Delcourt, Erickson and West showed that for any $n$-vertex graph $G$ there is a winning strategy for the cop on the graph $\gm$ obtained by replacing each edge of $G$ by a path of length $m$, if $m\geq n$ (Carragher et al., 2012). The present authors showed that, for all but a few small values of $n$, this bound may be improved to $m\geq n/2$, which is best possible (Haslegrave et al., 2016b). In this paper we consider the natural extension in which the cop probes a set of $k$ vertices, rather than a single vertex, at each turn. We consider the relationship between the value of $k$ required to ensure victory on the original graph with the length of subdivisions required to ensure victory with $k=1$. We give an asymptotically best-possible linear bound in one direction, but show that in the other direction no subexponential bound holds. We also give a bound on the value of $k$ for which the cop has a winning strategy on any (possibly infinite) connected graph of maximum degree $\Delta$, which is best possible up to a factor of $(1-o(1))$.
\end{abstract}

\section{Introduction}
Search games and pursuit games on graphs have been widely studied, beginning with a graph searching game introduced by Parsons \cite{Par76} in which a fixed number of searchers try to find a lost spelunker in a dark cave. The searchers cannot tell where the target is, and aim to move around the vertices and edges of the graph in such a way that one of them must eventually encounter him. The spelunker may move around the graph in an arbitrary fashion and at unlimited speed, and in the worst case may be regarded as an antagonist who knows the searchers' positions and is trying to escape them. Parsons \cite{Par78} subsequently introduced the \textit{search number}, $s(G)$, of a graph $G$, being the minimum number of searchers required to guarantee catching the spelunker. Megiddo, Hakimi, Garey, Johnson and Papadimitriou \cite{MHGJP} showed that computing the search number of a general graph is an NP-hard problem. From the subsequent proof by LaPaugh \cite{LaP93} that any number of searchers who can find the spelunker can do so while ensuring the spelunker cannot return to a previously searched edge, it follows that the decision problem is in NP (and hence NP-complete).

Robertson and Seymour \cite{RS83} introduced the concepts of path decompositions and the pathwidth, $\operatorname{pw}(G)$ of a graph $G$, which has deep connections to the theory of graph minors and to algorithmic complexity. Minor-closed families which have a forbidden forest are precisely those with bounded pathwidth \cite{RS83}, and many algorithmic problems which are difficult in general have efficient algorithms for graphs of bounded pathwidth \cite{Arn85}. Ellis, Sudborough and Turner \cite{EST83} independently defined the vertex separation number of a graph, and Kinnersley subsequently showed the equivalence of these two definitions \cite{Kin92}. Ellis, Sudborough and Turner showed that the search number of a graph is almost completely determined by its pathwidth, giving bounds of $\operatorname{pw}(G)\leq s(G)\leq\operatorname{pw}(G)+2$ \cite{EST94}.

Graph pursuit games date back to the classical Cops and Robbers game, introduced independently by Quillot \cite{Qui78} and Nowakowski and Winkler \cite{NW83} (who attribute it to G. Gabor). This involves one or more cops and a robber moving around a fixed graph. The cops move simultaneously and alternate moves with the robber, all moves being to neighbouring vertices. The cops win if one of them occupies the robber's location. On a particular graph $G$ the question is whether a given number of cops have a strategy which is guaranteed to win, or whether there is a strategy for the robber which will allow him to evade capture indefinitely. The \textit{cop number} of a graph is the minimum number of cops that can guarantee to catch the robber. An important open problem is Meyniel's conjecture, published by Frankl \cite{Fra87}, that the cop number of any $n$-vertex connected graph is at most $O(\sqrt{n})$.

Models which focus on finding an invisible target rather than catching a visible one have been the focus of much recent work. In the Hunter versus Rabbit game, studied by Adler, R\"acke, Sivadasan, Sohler and V\"ocking \cite{ARSSV} and the Cop versus Gambler game, studied by Komarov and Winkler, \cite{KW16} the aim is to catch a randomly-moving target as quickly as possible; in both cases the searcher is restricted to moving on the edges of the underlying graph, but the target is not. Related models can be used to design protocols for ad-hoc mobile networks \cite{CNS01}. In Hunter versus Rabbit, the rabbit's strategy is unrestricted; Adler et al.\ showed that the hunter can achieve expected capture time $O(n\log n)$, while for some graphs the rabbit can achieve expected survival time $\Omega(n\log n)$ \cite{ARSSV}. In Cop versus Gambler, the gambler's strategy is simply a probability distribution on the vertices, and his location at different time steps is independent. In this setting Komarov and Winkler showed that a cop who knows this probability distribution can achieve expected capture time $n$, which is trivially best possible for the uniform distribution, and a cop who does not know the distribution can still achieve capture time $O(n)$ \cite{KW16}.

In the problem variously referred to as Cat and Mouse, Finding a Princess, and Hunter and Mole, it is the target which is constrained to moving around a graph, and the searcher probes vertices one by one in an unrestricted manner. In between any two probes the target must move to an adjacent vertex. The searcher wins if she probes the target's location, but gets no information otherwise. A complete classification of graphs on which the searcher can guarantee to win was given by Haslegrave \cite{Has11,Has14} and subsequently and independently by Britnell and Wildon \cite{BW13} and Komarov and Winkler \cite{KW13}. 

In real-life searching we might gain information about how close an unsuccessful probe is. The simplest search model of this form is Graph Locating, independently introduced by Slater \cite{Sla75} and by Harary and Melter \cite{HM76}. In this model a set of vertices is probed and each probe reveals the graph distance to a stationary target vertex; the searcher wins if she can then determine the target's precise location. The minimum number of probes required to guarantee victory on the graph $G$ is its \textit{metric dimension}, $\mu(G)$. If the probed vertices are instead chosen sequentially, with each choice potentially depending on the results of previous probes, it may be possible to ensure victory with fewer probes; this is the Sequential Locating game, studied by Seager \cite{Sea13}.

In this paper we consider the Robber Locating game, introduced in a slightly different form by Seager \cite{Sea12} and further studied by Carragher, Choi, Delcourt, Erickson and West \cite{CCDEW}, as well as by the current authors \cite{HJKa, HJKb}. This combines features of the other games mentioned above. Like the Sequential Locating game, the aim is to deduce the target's location from distance probes, but, like Cops and Robbers or Cat and Mouse, the target moves around the graph in discrete steps. A single cop and robber take turns to act. For ease of reading we shall refer to the cop as female and the robber as male. The cop, who is not on the graph, can probe a vertex at her turn and is told the distance to the robber's current location. If at this point she can identify the robber's precise location, she wins. At the robber's turn he may move to an adjacent vertex. (The original version of the game also had the restriction that the robber may not move to the vertex most recently probed, but subsequent work has generally permitted such moves.)

In the Sequential Locating game the searcher can guarantee to win eventually, simply by probing every vertex, and the natural question is the minimum number of turns required to guarantee victory on a given graph $G$. In the Robber Locating game, by contrast, it is not necessarily true that the cop can guarantee to win in any number of turns. Consequently the primary question in this setting is whether, for a given graph $G$, the cop can guarantee to catch a robber who has full knowledge of how she will act. (On a finite graph, this is equivalent to asking whether there is some fixed number of turns in which she can guarantee victory.) We say that a graph is \textit{locatable} if she can do this and \textit{non-locatable} otherwise. 

The main result of Carragher et al.\ \cite{CCDEW} is that for any graph $G$ a sufficiently large equal-length subdivision of $G$ is locatable. Formally, write $\gm$ for the graph obtained by replacing each edge of $G$ by a path of length $m$, adding $m-1$ new vertices for each such path. Carragher et al.\ proved that $\gm$ is locatable whenever $m\geq \min\{\abs{V(G)},1+\max\{\mu(G)+2^{\mu(G)},\Delta(G)\}\}$. In most graphs this bound is simply $\abs{V(G)}$, and they conjectured that this was best possible for complete graphs, i.e.\ that $K_n^{1/m}$ is locatable if and only if $m\geq n$. The present authors showed that in fact $K_n^{1/m}$ is locatable if and only if $m\geq n/2$, for every $n\geq 11$ \cite{HJKa}, and then subsequently that the same improvement on the upper bound may be obtained in general: provided $\abs{V(G)}\geq 23$, $\gm$ is locatable whenever $m\geq\abs{V(G)}/2$ \cite{HJKb}. This bound is best possible, since $K_n^{1/m}$ is not locatable if $m=(n-1)/2$, and some lower bound on $\abs{V(G)}$ is required for it to hold, since $K_{10}^{1/5}$ is not locatable \cite{HJKa}. These results fundamentally depend on taking equal-length subdivisions, and do not imply any results for unequal subdivisions, since subdividing a single edge of a locatable graph can result in a non-locatable graph (as observed independently by Seager \cite{Sea14} and in \cite{HJKa}); however, the present authors showed that an unequal subdivision is also locatable provided every edge is subdivided into a path of length at least $2\abs{V(G)}$ \cite{HJKb}.

\section{Multiple probes vs subdivisions}

Subdividing the edges of $G$ tends to favour the cop both by slowing down the robber's movement around the graph and by giving her extra vertices to probe. Another way in which we might make the game easier for the cop is to allow her to choose a set of $k$ vertices to probe at each turn, rather than a single vertex. Formally, we define the $k$-probe version of the game as follows. At the cop's turn she chooses a list $\seq uk$ of vertices, and is then told the corresponding list $d(v,u_1),\ldots,d(v,u_k)$ of distances to the robber's location, $v$. (Note that the cop must specify all $k$ vertices before being told any of the distances.) If, from this and previous information, she can deduce $v$, she wins. At the robber's turn, as before, he may move to a neighbouring vertex. We say $G$ is $k$-locatable if the cop can guarantee to win the $k$-probe version of the game, that is to say if she has a deterministic strategy which will succeed against any possible sequence of moves for the robber. 

The game thus provides two natural graph invariants. Write $\rlp(G)$ for the minimum value of $k$ such that $G$ is $k$-locatable, $\rls(G)$ for the smallest value such that $\gm$ is locatable whenever $m\geq\rls(G)$. (We use the subscripts p and s to indicate that the two quantities are the minimum number of probes and subdivisions respectively required for the cop to win.) We investigate the relationship between the two, showing that $\rls(G)\leq(2+o(1))\rlp(G)$ (the factor of $2$ is best possible), but that no subexponential bound holds in the other direction. In Section~\ref{delta} we also bound $\rlp(G)$ for connected, but not necessarily finite, graphs of maximum degree $\Delta$; our bound is best-possible up to a lower order error term, as shown by the infinite regular tree. We will always assume that the graph $G$ is connected; provided there are only countably many components the cop can always reduce the problem to the connected case by probing components one by one until the component containing the robber is identified.

Recall that $\gm$ is the graph obtained by replacing each edge of $G$ with a path of length $m$ through new vertices. Each such path is called a \textit{thread}, and a \textit{branch vertex} in $\gm$ is a vertex that corresponds to a vertex of $G$. We write $u\thrd v$ for the thread between branch vertices $u$ and $v$.  We use ``a vertex on $u\cdots v$'' to mean any of the $m+1$ vertices of the thread, but ``a vertex inside $u\cdots v$'' excludes $u$ and $v$. When $m$ is even we use the term ``midpoint'' for the central vertex of a thread, and when $m$ is odd we will use the term ``near-midpoint'' for either of the vertices belonging to the central edge of a thread. We say that vertices $\seq vr$ form a \textit{resolving set} for $W\subseteq V(G)$ if the vectors $(d(w,v_1),\ldots,d(w,v_r))$ are different for all $w\in W$.

\medskip
We first note that $\rlp(G)$ and $\rls(G)$ do not exist unless $G$ is countable. Even if $G$ is countable they may not exist, as shown by the infinite clique. 
\begin{lem}\label{countable}For any integer $k$, if $G$ is $k$-locatable then $V(G)$ is countable.\end{lem}
\begin{proof}Suppose $G$ is $k$-locatable, and fix a winning strategy $\mathcal S$ for the cop. For each vertex $v$ let $t_v$ be the number of turns taken for 
strategy $\mathcal S$ to succeed against a robber who is stationary at $v$, and let $\boldsymbol s_t^{(v)}$ be the vector of distances returned in the cop's $t$th turn, for $t=1,\ldots,t_v$. Then $(\boldsymbol s_1^{(v)},\ldots\boldsymbol s_{t_v}^{(v)})$ are finite sequences with terms in $\mathbb Z^k$, and so the set of possible sequences is countable. Since each sequence determines the robber's location, they must all be distinct, so $V(G)$ is countable.\end{proof}
\begin{rem}If $G$ is uncountable then $G^{1/m}$ is uncountable for any $m$, so neither $\rlp(G)$ nor $\rls(G)$ exist. In fact, since we only consider a stationary robber, we have proved the stronger statement that there is no winning strategy for Seager's Sequential Locating game \cite{Sea13} on any uncountable graph.\end{rem}

\subsection{$\rls(G)$ is $O(\rlp(G))$}

In this section we show that $\rls(G)\leq 2\rlp(G)+2$ whenever $\rlp(G)$ exists, and also show that this bound cannot be improved to $a\rlp(G)+b$ for any $a,b\in\mathbb R$ with $a<2$ by giving examples of graphs $G_n$ with $\rlp(G_n)=n$ and $\rls(G_n)\geq 2n-O(\log(n))$. 

Note that $\rlp(G)\leq\mu(G)$, since if the cop probes $\mu(G)$ vertices which form a resolving set for $V(G)$ she will locate the robber immediately. Thus \theorem{linear} implies that $\rls(G)\leq 2\mu(G)+2$, a considerable improvement of the bound $\rls(G)\leq 1+\max\{\mu(G)+2^{\mu(G)},\Delta(G)\}$ given by Carragher et al.\ \cite{CCDEW}.

\begin{thm}\label{linear}If $G$ is $k$-locatable then $\gm$ is locatable provided $m\geq 2k+2$.\end{thm}
\begin{proof}Suppose $G$ is $k$-locatable (and so, by \lemma{countable}, $V(G)$ is countable), and fix a winning strategy $\mathcal S$ in which the cop probes $k$ vertices of $G$ at each turn. We show how to use $\mathcal S$ to define a winning strategy for the cop with one probe per turn on $\gm$. In fact we can do this with the added restriction that the cop only probes branch vertices of $\gm$. Note that from the result of such a probe considered mod $m$ the cop can always determine the robber's distance to his nearest branch vertex; in particular, she can determine whether he is at a midpoint or near-midpoint, and whether he is at a branch vertex. Also, provided the robber is not at a midpoint, she can determine the distance between his nearest branch vertex and the branch vertex probed (being the nearest multiple of $m$ to the result of the probe), and hence the distance between the corresponding vertices of $G$.
 
\begin{clm}From any position, the cop can probe branch vertices of $\gm$ such that after at finitely many turns either she locates the robber or he reaches a midpoint or near-midpoint.\end{clm}
\begin{poc}If the robber does not reach a midpoint or near-midpoint, his closest branch vertex, $v$, cannot change. By probing branch vertices in order, the cop will eventually identify $v$. Then she starts probing branch vertices in order again. If the robber reaches $v$ she will recognise that he is at a branch vertex, which must be $v$, and win. If not then he will remain on a single thread $v\thrd w$, and once the cop probes $w$ she will win.\end{poc}
\begin{clm}Fix a set $\{\seq ak\}$ of branch vertices. Suppose the robber is at a midpoint or near-midpoint. Then the cop may probe branch vertices of $\gm$ such that after finitely many turns either she has won or all of the following hold:
\begin{enumerate}[(a)]
\item the robber is at some branch vertex, $v$;
\item he has not reached a branch vertex in the interim; and
\item she has probed every $a_i$ after the robber's last visit to a (near-)midpoint, and hence while the robber was closer to $v$ than to any other branch vertex.
\end{enumerate}\end{clm}
\begin{poc}The cop proceeds as follows. She starts by probing any branch vertex, and every time the result of a probe indicates that the robber was at a (near-)midpoint, she probes a new branch vertex (using an ordering which includes every branch vertex). Once the robber is no longer at a (near-)midpoint she starts probing $\seq ak$ in turn. Every time that the robber returns to a (near)-midpoint she resumes probing new branch vertices, and every time that he leaves the (near-)midpoint(s) she restarts probing $\seq ak$, beginning with $a_1$. If she finishes probing $\seq ak$ she resumes probing new branch vertices. She continues this process until either she has won or the robber is at a branch vertex.

In this way, the cop will probe at least one new branch vertex every $k+1$ turns, so she will eventually probe both endpoints of the robber's thread, unless he reaches a branch vertex first. But if she probes both endpoints while the robber remains inside a thread, she identifies the thread and so wins. If the robber reaches a branch vertex first, then since $m\geq 2k+2$ the cop will have probed at least $k+1$ vertices since the robber was last at a (near-)midpoint, and therefore she will have probed all of $A$ in that time, as required.\end{poc}

Write $m_1$ for the first (near-)midpoint the robber reaches, $v_1$ for the next branch vertex he reaches after leaving $m_1$, $m_2$ for the next (near-)midpoint after leaving $v_1$, and so on. Note that $v_{i+1}$ is either equal to $v_i$ or adjacent to it in $G$, and so $v_1v_2\cdots$ is a possible trajectory for the robber in the $k$-probe game on $G$. We know that strategy $\mathcal S$ locates the robber in that game; write $A_i$ for the set of vertices probed by the cop at turn $i$ when playing strategy $\mathcal S$ against a robber following trajectory $v_1v_2\cdots$.

The cop will alternate between using Claim~1.1 to force the robber to $m_i$ and using Claim~1.2 with set $A_i$ to force him to $v_i$. We show that, assuming she has not yet won, she has enough information to do this by induction: she does for $i=1$ because $A_1$ is fixed; for $i>1$, $A_i$ depends only on the distances (in $G$) of vertices in $A_j$ to $v_j$ for $j<i$, which the cop will be able to deduce using (c) above. In this manner she ensures for each $i$ that either the robber reaches $v_i$ in finite time or he is caught before reaching it. 

If the cop would have located the robber on her $t$th turn playing strategy $\mathcal S$ against a robber following trajectory $v_1v_2\cdots$ on $G$, then  she has enough information to identify $v_t$ before the robber leaves it, and so locates him.\end{proof}

In fact the factor of $2$ in \theorem{linear} is best possible, as there are $k$-locatable graphs for which subdivisions of length at least $(2-o(1))k$ are required. We use the notation $[n]$ for the set $\{1,\ldots,n\}$ and $\level nk$ for the set of $k$-element subsets of $[n]$. Define the graph $G_{n,k}$, where $1\leq k<n$ as follows. Take as a vertex set $\{v_i\mid i\in[n]\}\cup\{w_A\mid A\in\level nk\}$; for each $A\neq B\in\level nk$, add the edge $w_Aw_B$, and for each $i\in[n]$ and $A\in\level nk$ such that $i\not\in A$, add the edge $v_iw_A$. Figure~1 shows $G_{4,2}$.

\begin{figure}[h]
\begin{center}
\begin{tikzpicture}[scale=0.9]
\tikzstyle{vertex}=[draw, shape=ellipse, minimum size=5pt, fill=black, inner sep=0pt]
\foreach \x/\y/\name in {5/0/A, 4/1/B, 6/1/C, 4/3.5/D, 6/3.5/E, 5/4.5/F}
{\node[vertex] (\name) at (\x, \y) {};}
\foreach \x/\y in {}
{\node[vertex] at (\x, \y) {};}
\foreach \from/\to in {A/B, A/C, A/D, A/E, A/F, B/C, B/D, B/E, B/F, C/D, C/E, C/F,
D/E, D/F, E/F}
{\draw (\from) to (\to);}
\foreach \x/\y/\label in {5/-0.4/$w_\text{\{1, 2\}}$, 3.4/0.8/$w_\text{\{1, 4\}}$,
6.6/0.6/$w_\text{\{1, 3\}}$, 3.1/3.4/$w_\text{\{2, 3\}}$, 6.6/3.7/$w_\text{\{2,
4\}}$, 5/4.9/$w_\text{\{3, 4\}}$}
{\node at (\x, \y) {\Large \label};}
\foreach \y/\name in {0/4, 1.5/3, 3/2, 4.5/1}
{\node[vertex, color=red] (\name) at (-3, \y) {};
\node at (-3.5,\y) {\Large $v_{\name}$};}
\foreach \from/\to in {1/D, 1/E, 1/F, 2/B, 2/C, 2/F, 3/A, 3/B, 3/E, 4/A, 4/C, 4/D}
{\draw[color=red] (\from) to (\to);}
\draw (-3, 2.25) ellipse (2cm and 4cm);
\draw (5, 2.25) ellipse (3cm and 4cm);
\end{tikzpicture}
\end{center}
\caption{$G_{4, 2}$}
\end{figure}

\begin{lem}Provided $k\leq n-2$, $\rlp\bigl(G_{n,k}\bigr)=n-1$.\end{lem}
\begin{proof}Consider the effect of probing $\seq v{n-1}$. If the robber is at one of those vertices, he has certainly been located. If he is at $v_n$, each probe will return distance $2$, whereas if he is at some $w_A$ at least one probe will return $1$, so in the former case he is located. If he is at some $w_A$ then from the distances to each of $\seq v{n-1}$ the cop can deduce the distance to $v_n$, since exactly $k$ of these $n$ distances must be $1$, and so she can deduce whether $i\in A$ for each $i$, and hence determine $A$.

To complete the proof we show that the graph is not $(n-2)$-locatable. This is true even if the robber is confined to the set $W=\{w_S\mid S\in\level nk\}$: we show that for any set of $n-2$ probes there is some $A\neq B\in\level nk$ such that the probes fail to distinguish $w_A$ and $w_B$. This is certainly true if none of the vertices probed are among the $v_i$, since then there are at least two unprobed vertices in $W$, and these have distance $1$ from every probed vertex. Suppose that exactly $r>0$ of the $v_i$ and $n-r-2$ vertices in $W$ are probed, and let $t=\min\{k-1,r\}$. There are then $\binom{n-r}{k-t}\geq n-r$ vertices in $W$ which are adjacent to the first $t$ of those $r$ vertices, but not to the remaining $r-t$. At least two of these vertices are unprobed, and these two vertices have the same distance as each other from every probed vertex.

Consequently, at each of the cop's turns she must leave some two vertices in $W$ undistinguished, and both of these are possible locations for the robber since from any vertex in $W$ he can reach either of them. So the cop cannot guarantee to locate the robber at any point.\end{proof}

\begin{lem}Provided $\binom{n-k-\floor{m/2}}k\geq 2m+2$, $\rls\bigl(G_{n,k}\bigr)>m$.\end{lem}
\begin{proof}We'll play the game on $G_{n,k}^{1/m}$ with some restrictions. The robber must stay within the $w$ section (\ie $W$ together with all threads between vertices in $W$), and every time he gets to a branch vertex he must leave it at his next turn and proceed along some thread without stopping or turning round. Also, every time he leaves a branch vertex he will announce which one it was. We show that the cop cannot guarantee to identify the branch vertex the robber is approaching by the time he reaches it, and thus she cannot guarantee to locate the robber. When the robber leaves a branch vertex, say $w_A$, the cop has no information about the next branch vertex he is approaching (call this $w_B$). During the next $\ceil{m/2}$ probes, she can identify for each $i\in A$ whether or not $i\in B$, by probing $v_i$ or inside one of the threads leading from it, and she can additionally eliminate two possible candidates for $B$ per turn, by probing inside a thread of the form $w_C\thrd w_D$. Other probes are not helpful: probing $v_i$ for $i\not\in A$ gives no information, since the shortest path to the robber passes through $w_A$, and probing inside a thread meeting such a $v_i$ only eliminates one candidate for $B$. In the remaining $\floor{m/2}$ probes up to and including the time the robber is at $w_B$, she can determine whether $i\in B$ for $\floor{m/2}$ other values of $i$, and eliminate at most two other candidates for $B$ per turn. If each of the at most $\floor{m/2}+k$ values of $i$ the cop checks are not in $B$, and none of the at most $2m$ candidates for $B$ she tests are correct, there are at least $\binom{n-k-\floor{m/2}}k-2m$ possibilities for $B$ remaining. If this is at least $2$, she cannot guarantee to catch the robber before he next leaves a branch vertex.\end{proof}

Suppose $m=2\ceil{n-a\log n}$ and $k=\floor{b\log n}$. Then 
\begin{align*}\binom{n+1-k-m/2}{k}&\geq\binom{\ceil{(a-b)\log n}}{\floor{b\log n}}\\
&=\biggl(\bfrac{a-b}b^b+o(1)\biggr)^{\log n}\,.\end{align*}
If $\bfrac{a-b}b^b>\mathrm e$, as is the case when $a=3.6$ and $b=0.8$, then $\binom{{n+1}-k-m/2}{k}=\omega(n)$ and so for sufficiently large $n$ 
we have $\rls\bigl(G_{[n+1],k}\bigr)>m$. Thus we have established the following result.

\begin{thm}For all sufficiently large $n$ there is a finite graph $G$ with $\rlp(G)=n$ and $\rls(G)>2n-7.2\log n$.\end{thm}
\subsection{$\rlp(G)$ is not $O(\rls(G))$}
In this section we show that $\rlp(G)$ can be exponentially large in terms of $\rls(G)$. Naively we might expect a linear bound, since a successful
strategy for $\gm$ may only have $m$ turns between visits of the robber to different branch vertices. However, each vertex probed can depend on the results of previous ones, so the number of vertices which the strategy could potentially probe in those $m$ turns can be large. 

Construct a graph $G_n$ as follows. There are four classes of vertices: $A, B, C$ and $D$. Each has $2^n$ vertices. Label the 
vertices in $B$ as $b0...0, \ldots, b1...1$ ($b$ followed by all binary words of length $n$), and similarly for $C$. Label the vertices 
of $A$ as $a, a1, a01, a11 \ldots$ ($a$ followed by all words of length at most $n$ which are either empty or end in 1), and 
similarly for D. The only edges are between adjacent classes: $A$--$B$, $B$--$C$ or $C$--$D$. All vertices in $B$ are adjacent to all in $C$. 
The vertices $ax$ and $by$ are adjacent if and only if $x$ is a prefix of $y$; similarly for $dx$ and $cy$. Figure~2 shows $G_2$.

\begin{figure}
\begin{center}
\begin{tikzpicture}[scale=0.8]
\tikzstyle{vertex}=[draw, shape=ellipse, minimum size=5pt, fill=black, inner sep=0pt]
\foreach \x/\y/\name in {0/0/A1, 0/2/A2, 0/4/A3, 0/6/A4, 4/0/B1, 4/2/B2, 4/4/B3,
4/6/B4, 7/0/C1, 7/2/C2, 7/4/C3, 7/6/C4, 11/0/D1, 11/2/D2, 11/4/D3, 11/6/D4}
{\node[vertex] (\name) at (\x, \y) {};}
\foreach \from/\to in {A1/B1, A1/B2, A1/B3, A1/B4, A2/B3, A2/B4, A3/B2, A4/B4,
D1/C1, D1/C2, D1/C3, D1/C4, D2/C3, D2/C4, D3/C2, D4/C4}
{\draw[color=red] (\from) to (\to);}
\foreach \from in {B1, B2, B3, B4}
{\foreach \to in {C1, C2, C3, C4}
{\draw (\from) to (\to);}}
\foreach \x in {-0.2, 4, 7, 11.2}
{\draw (\x, 3.25) ellipse (1cm and 4.5cm);}
\node at (-0.2, 7.1) {\Large $A$};
\node at (4, 7.1) {\Large $B$};
\node at (7, 7.1) {\Large $C$};
\node at (11, 7.1) {\Large $D$};
\foreach \x/\y/\name in {-0.4/0/a, -0.4/2/a1, -0.5/4/a01, -0.5/6/a11, 11.4/0/d,
11.4/2/d1, 11.5/4/d01, 11.5/6/d11, 4/-0.4/b00, 4/1.6/b01, 3.8/4.4/b10, 4/6.4/b11,
7/-0.4/c00, 7/1.6/c01, 7.2/4.4/c10, 7/6.4/c11}
{\node at (\x, \y) {\Large $\name$};}
\end{tikzpicture}
\end{center}
\caption{$G_2$}
\end{figure}

\begin{lem}For each $n\geq 2$, $\rls(G_n)=n+1$.\end{lem}
\begin{proof}The cop's strategy on $G_n^{1/m}$ for $m\geq n+1$ is as follows. First she probes branch 
vertices one by one until either a probe results in a multiple of $m$, indicating the robber is at a branch vertex, or two different probes have
given responses of less than $m$. If the latter happens first then the robber has not passed through a branch vertex and the cop has
identified both ends of the thread he is on, so has located him.

Suppose that a probe has given a multiple of $m$, indicating that the robber is at a branch vertex. From the result of the probe mod $2m$, the cop will also know either that he is in $A\cup C$ or that he is in $B\cup D$; assume without loss of generality 
the former. Probing $d$ next will tell her whether the robber was in $A$, was in $C$ and is still there, or was in $C$ and has moved. 
In the latter case she next probes at $a$, which will tell her 
whether the robber has returned to $C$ or not, and if not in whether he is on a thread between $B$ and $C$ or between $C$ and $D$. Now she probes branch vertices until the robber is at a branch vertex or is caught; if he reaches a branch vertex before being caught, the cop will know, from the information she had about his thread together with the result of her last probe mod $2m$, that he is at a branch vertex in a specific class. 

\begin{case}The result of a probe establishes that the robber is in $A$, or establishes that the robber is in $D$.\end{case}
Without loss of generality we assume the cop knows that the robber is in $A$. She probes vertices in $A$ which are possible locations 
for the robber, until she gets a response which is not a multiple of $m$, indicating that he has left $A$, or a response of $0$, which locates him. One of these must happen, since if he remains at his location in $A$ she will eventually probe it.

Once the robber moves away from $A$ he must be on a thread between a vertex in $A$ and one in $B$; the cop now
attempts to identify the endpoint in $B$. She does this by first probing $a1$ to find out whether it is of the form $b1x$ or $b0x$: a response less than $2m$ indicates that it is of the form $b1x$; a response greater than $2m$ indicates that it is of the form $b0x$, and a response of exactly $2m$ (or $0$) indicates that he has returned to $A$. 

If the robber does not return to $A$, the cop has identified the first digit of the endpoint of his thread which is in $B$. She then probes either $a11$ or $a01$, depending on this digit, in order to determine the second digit. By continuing in this manner, either she will identify that the robber has returned to his original vertex in $A$ or she will identify the endpoint of his thread in $B$ while he is still inside the thread. In the latter case she then probes 
vertices in $A$ in turn, allowing her to tell if he reaches the branch vertex in $B$ or to eventually find the other end of his current thread 
if he remains inside it. If he returns to $A$ she restarts this case; since the robber must have returned to the same branch vertex in $A$ she reduces the set of possible locations in $A$ with each iteration of this case, so she will locate him within $2^n$ iterations. 

\begin{case}The result of a probe establishes that the robber is in $B$, or establishes that the robber is in $C$.\end{case}
Without loss of generality we assume the cop knows that the robber is in $A$. She probes vertices at distance $1$ from $B$ which are on threads for which one endpoint is a possible location for the robber in $B$ and the other is $c0...0$. She does this until she gets a response which is not $\pm1$ mod $m$, indicating that the robber has moved. If he does not move she will eventually probe a vertex adjacent to him and win.

Once the robber leaves $B$, the cop will know whether he is on a thread leading to $c0...0$, since she will get a response of $0$ or $2m-2$ in this case and $2$ or $2m$ otherwise; since $m\geq n+1\geq 3$, $2m-2\neq 2$. If the robber is not on such a thread, the cop attempts to either identify the vertex in $C$ which is an endpoint of his thread or identify that no such vertex exists (\ie he is on a thread between $A$ and $B$). She does this as in Case~1 by probing first $d1$, then $d01$ or $d11$, as appropriate, and so on. If she detects that the robber has returned 
to $B$ (indicated by a result of exactly $2m$), she restarts this case; as before she reduces the set of possible locations in $B$ every time this happens, so at most $2^n$ iterations can occur. 

If the robber does not return to $B$, and is on a thread between $B$ and $C$, the cop will have time to complete this process while he is still on that thread, and since the endpoint in $C$ is not $c0...0$, she will have probed at least one vertex at distance less than $2m$. Thus she will have identified both that he is on a thread between $B$ and $C$ and the endpoint of that thread in $C$, either while he is at that branch vertex or while he is still inside the thread. In the latter case she probes vertices in $B$ until she locates the other end of the thread or establishes that the robber has reached the known vertex in $C$. 

If the robber does not return to $B$, and is on a thread between $A$ and $B$, again the cop will complete the process before he reaches $A$, and since in this case none of her probes will have given a result of less than $2m$, she will identify that he is on a thread between $A$ and $B$. If he has not reached $A$ when she identifies this, she probes vertices in $B$ in turn until she either locates him or identifies that he is in $A$, reducing to Case~1. 

This completes the proof that $G_n^{1/m}$ is locatable for $m\geq n+1$.

\medskip
Next we show that $G_n^{1/n}$ is not locatable. Again, we may restrict the robber: he may only visit branch vertices in $\{a,d\}\cup B\cup C$; each time he leaves a branch vertex he announces which vertex he has just left, and moves directly along a thread towards the next branch vertex. However, he may choose to remain stationary at a branch vertex. While the robber is permitted to move to $a$ or $d$, he will never actually do this; however, sometimes there will only be two possible locations which are consistent with the probe results, one of which is $a$.

Suppose the robber is at a branch vertex in $B\cup\{d\}$ (the case $C\cup\{a\}$ is equivalent), but the cop does not know which one (\ie there are at least two possibilities). The robber might remain at this branch vertex, so the cop needs to probe some vertex which eliminates at least one possible branch vertex. So she must probe a vertex in $B\cup\{d\}$, or in $A$, or inside a thread between $A$ and $B$ or between $B\cup\{d\}$ and $C$. Suppose when she does this that the robber has just left a branch vertex, which is not $d$ (since there were at least two possibilities, this is always possible). He is heading towards a vertex in $C\cup\{a\}$. If she has just probed a vertex in $B$ she has no further information about where the robber is heading. If she has just probed $d$, or a vertex in $A$ or between $A$ and $B$, she may be able to tell whether or not he is heading towards $a$, but cannot distinguish destinations in $C$. If she has just probed a vertex in $C$ or between $B$ and $C$ she can tell whether he is heading to a specific vertex in $C$, but cannot distinguish other destinations in $C\cup\{a\}$. Finally, if she has just probed inside the thread $cx\thrd d$ for some $x$ then she can tell whether he is heading to $cx$, but cannot distinguish other vertices in $C\cup\{a\}$ (the shortest route from anywhere inside this thread to the vertex adjacent to $by$ on the thread to $a$ is via $by$). Consequently, some response to the probe is consistent with at least $2^n$ possible destinations.

In order to win before the robber reaches the next branch vertex, the cop must identify the thread he is on by the time he reaches the end of it, so she has $n-1$ turns remaining to do this. If she probes a vertex in $A\setminus\{a\}$ or $B$, or on a thread between the two, she gets no additional information. If she probes a vertex in $C\cup\{a\}$ or on a thread between $B$ and $C\cup\{a\}$ then she may eliminate one possible thread, but will not distinguish between the rest. If she probes a vertex on a thread between $C$ and $D$, or a vertex in $D$, then she may eliminate one possible thread, but all other possible locations for the robber will give one of two possible responses. Thus, if there were $k$ possibilities before the probe, and the response to the probe is that consistent with the greatest number of those possibilities, at least $\ceil{(k-1)/2}$ possibilities remain. Since there were $2^n$ possibilities, at least $2$ will remain after $n-1$ additional probes, and so the cop cannot guarantee to locate him.\end{proof}

\begin{lem}For each $n\geq 1$, $\rlp(G_n)=2^n$.\end{lem}
\begin{proof}We show that if the cop is permitted $2^n-1$ probes per turn, the robber can escape forever provided at every turn he knows which vertices
the cop will probe at her next turn. Suppose he constrains himself to moving between $B$ and $C$. If he is in $B$, the cop will 
get no information about which vertex in $B$ he is at unless she probes at least one vertex in $A\cup B$, so he need not move unless that happens. 
Similarly if he is in $C$ he need not move unless the cop is about to probe at least one vertex in $C\cup D$. If he moves, he can get to 
any vertex in the other class, so in order to win the cop must at some point probe a resolving set for $B$, say, together with some vertex in 
$C\cup D$. Since the smallest resolving set for $B$ has size $2^n-1$, and another probe is needed, this is not possible.

If the cop is permitted $2^n$ probes per turn, she may win as follows. In the first turn she probes the set $\{ax,dx\mid x\text{ has length}<n\}$. 
This will either locate the robber immediately or pin him down to a set of the form $\{by0, by1\}$ or $\{cy0, cy1\}$; assume without loss of generality he is in $\{by0, by1\}$.
On her next turn the cop probes $\{ay1\}\cup D\setminus\{d\}$. This either locates the robber, or shows that he is in $A\cup\{c0...0\}$. If the robber is in $A\cup\{c0...0\}$, next turn he must be in $A\cup B\cup\{c0...0,d\}$. Now the cop can win by probing $A$, since it is a resolving set for $A\cup B\cup\{c0...0,d\}$.\end{proof}

Thus we have shown the following result.
\begin{thm}For every $m\geq 3$ there is a finite graph $G$ with $\rls(G)=m$ and $\rlp(G)=2^{m-1}$.\end{thm}

\section{Graphs of bounded degree}\label{delta}

In this section we obtain a general bound for $\rlp(G)$ (and hence, using \theorem{linear}, a bound of the same order on $\rls(G)$) in terms of the maximum degree $\Delta(G)$. We allow graphs to be infinite (but connected) in this case. By considering infinite regular trees we show that our bound is tight up to a factor of $1+o(1)$. In the case $\Delta=3$ we show that $\rlp(G)\leq3$, which is best possible. The case $\Delta(G)=2$ (\ie $G$ is a finite cycle or path, or infinite ray or path) is trivial: any such graph is $2$-locatable, since any two adjacent vertices form a resolving set for $V(G)$, and this is best possible since $C_3$ is not $1$-locatable.

\subsection{General quadratic bounds}
In this section we give a quadratic upper bound on $\rlp(G)$ in terms of $\Delta(G)$, and a lower bound on $\max\{\rlp(G)\mid\Delta(G)=\Delta\}$ which differs from our upper bound only in lower-order terms.

\begin{thm}For any connected graph $G$ with $\Delta(G)=\Delta$, $\rlp(G)\leq\floorf{(\Delta+1)^2}{4}+1$.\end{thm}
\begin{proof}We give a winning strategy for the cop using $\floorf{(\Delta+1)^2}{4}+1$ probes at each turn. On her first turn she probes arbitrary vertices; since $G$ is connected, all distances are finite. 

Suppose that, from the results of probes at the cops $t$th turn, she knows that the robber is at distance $d_t$ from some vertex $v_t$, and that the shortest path from $v_t$ to his location passes through one of $k_t$ neighbours of $v_t$, $\seq w{k_t}$. For each $i$ choose $\Delta-k_t$ neighbours of $w_i$, not including $v_t$. At her next turn the cop probes all these neighbours, together with $v_t$ and $\seq w{k_t}$. This is a total of at most
\[
k_t(\Delta-k_t+1)+1\leq\floorf{(\Delta+1)^2}{4}+1
\]
vertices, so she can always do this. If none of the distances returned is less than $d_t$, then at least one of the $w_i$ (say $w_1$) will return exactly $d_t$, and the shortest path from $w_1$ to the robber's location must not pass through $v_t$ or any of $\Delta-k_t$ other neighbours of $w_1$. So setting $v_{t+1}=w_1$, the cop is in the same position as before, with $d_{t+1}=d_t$ and $k_{t+1}<k_t$. Consequently $(d_t,k_t)$ is decreasing in the lexicographic ordering; since $k_t\leq\Delta$, the cop takes a bounded number of steps to catch the robber from any particular set of responses to the initial probe.
\end{proof}
\begin{thm}For some connected graph $G$ with $\Delta(G)=\Delta$, $\rlp(G)\geq\floorf{\Delta^2}{4}$.\end{thm}
\begin{proof}We show that on the infinite $\Delta$-regular tree, $T_{\Delta}$, if the cop probes fewer than $\floorf{\Delta^2}{4}$ vertices on each turn, it is possible that she never probes a vertex within $r$ of the robber's location, where $r$ is an arbitrarily large distance. In fact we claim that for any fixed $r$, it is possible that for every $t$, after the cop's $t$th turn there is some vertex $v_t$ such that $T_{\Delta}-v_t$ has at least $\ceilf{\Delta-1}{2}$ components which have never been probed, the robber's distance from $v_t$ is $r$, and any vertex at distance $r$ from $v_t$ in the unprobed components is possible. This is certainly possible after the first step, since there is some finite subtree containing all the vertices probed, so setting $v_1$ be a leaf of that subtree, there are $\Delta-1$ components of $T_{\Delta}$ which have not been probed, and all vertices in these components which have distance $r$ from $v_1$ are possible locations for the robber which are not distinguished by the results of the cop's first turn. Suppose that the cop is in the required situation after her $t$th turn, and write $\seq wk$ for the neighbours of $v_t$ in unprobed components of $T_{\Delta}-v_t$. If, at the cop's $(t+1)$th turn, she probes vertices in at least $\floorf{\Delta+1}{2}$ of the components of $T_{\Delta}-w_i$ which do not include $v_t$ for each $i$, then she makes at least
\[
\ceilf{\Delta-1}{2}\floorf{\Delta+1}{2}=\floorf{\Delta^2}{4}
\]
probes, since all these sets of vertices are disjoint. So this is not the case, and without loss of generality $T_{\Delta}-w_1$ has at least $\ceilf{\Delta-1}{2}$ unprobed components. Provided the robber was in one of these, and has moved further away from $v_t$, the same situation holds with $v_{t+1}=w_1$.\end{proof}
\subsection{Exact result for maximum degree $3$}
In this case we prove that all connected graphs with maximum degree $3$ are $3$-locatable. This is trivially best possible, since $\rlp(K_4)=3$. Again, our bound applies even if the graph is infinite.  

\begin{thm}For any connected graph $G$ with $\Delta(G)=3$, $\rlp(G)\leq 3$.\end{thm}
\begin{proof}
If $G\cong K_{3,3}$, say with vertices $a,b,c$ in one class and $u,v,w$ in the other, the cop can win by probing $a,b,u$ on the first turn and, if necessary, $a,b,v$ on the second. Henceforth we assume $G\not\cong K_{3,3}$. We will also use the following observation several times.

\begin{obs}\label{square}Suppose the cop knows that at the time of her last probe the robber was at one of two vertices $p,q$ which have at least two neighbours $r,s$ in common. Write $p'$ for the third neighbour of $p$, and $q'$ for the third neighbour of $q$. If $p'=q$, or $p'$ is not adjacent to both $r$ and $s$, the cop can win by probing $q,r,s$. Otherwise $q'$ is not adjacent to both $r$ and $s$, so she can win by probing $p,r,s$.\end{obs}

We give a winning strategy for the cop. On her first turn she probes any three vertices, and since the graph is connected the robber's distance to each is finite. Suppose she has just probed a vertex and gotten a result of $r$. In the case $r\geq 2$, we show how she may get a result of less than $r$ from some probe within the next three time steps, and so by repeatedly employing this tactic she may eventually force a result of $1$. We then show how she may win from that position.

Suppose the cop probed $x_0$ which returned distance $r$. At the next turn she probes the neighbours of $x_0$; one of these probes must be at distance at most $r$ from the robber's new position. If it is less than $r$, we are done; otherwise from probing $x_1$, say, we get a result of $r$, and we know that the shortest path from $x_1$ to the robber's location must not go through $x_0$. Write $x_2,x_3$ for the other neighbours of $x_1$. Let $S=(\Gamma(x_2)\cup\Gamma(x_3))\setminus\{x_1,x_2,x_3\}$. At her next turn the cop probes as many vertices in $S$ as possible. If none of the responses is less than $r$ then we must have $\abs{S}=4$, and the one unprobed vertex of $S$ would have returned less than $r$. Now she probes all the neighbours of that vertex; one of them must be at distance less than $r$.

Continuing in this manner, the cop will either win or reach a position where she has just probed some vertex $v$ adjacent to the robber's position. We now show how the cop can win from this point.

Write $a,b,c$ for the neighbours of $v$. The cop next probes $a,b,c$. It is not possible for there to be three locations for the robber consistent with the results of these probes, since $G\not\cong K_{3,3}$, so either the cop has won or there are exactly two possible locations for the robber. Consequently, if at least two of the probes returned $1$, she can win using \observ{square}. 

The only remaining possibility is that exactly one of the probes, say at $a$, returned $1$, and the robber could have been at either of the neighbours of $a$ which are not $v$, say $u$ and $w$. If $u$ and $w$ have another common neighbour, the cop can win using \observ{square}. If not, we may write $\Gamma(u)=\{a,x_1,x_2\}$ and $\Gamma(w)=\{a,y_1,y_2\}$, where $x_1,x_2,y_1,y_2$ are all distinct.

On her next turn, the cop will probe at $a$ and two of $x_1,x_2,y_1,y_2$; we will describe how she chooses which two to probe below. Note that the probe at $a$ will distinguish whether the robber is at $a$, in $\{u,v\}$ or in $\{x_1,x_2,y_1,y_2\}$, and the other probes will distinguish between $u$ and $v$, so the cop will win unless she is unable to distinguish between the two unprobed vertices in $\{x_1,x_2,y_1,y_2\}$ (and the robber is at one of them). We consider five cases depending on the local structure.

\begin{enumerate}[(i)]
\item One of $x_1,x_2,y_1,y_2$ (say $x_1$) is adjacent to exactly one of the others. In this case the cop probes $a$, $x_1$ and one of $x_2,y_1,y_2$ which is not adjacent to $x_1$.
\item One of $x_1,x_2,y_1,y_2$ (say $x_1$) is adjacent to two of the others. In this case the cop probes $a$, $x_1$ and one of $x_2,y_1,y_2$ which is adjacent to $x_1$.
\item Some pair of $x_1,x_2,y_1,y_2$ (say $x_1,y_1$) are at distance more than $2$. In this case the cop probes $a$, $x_1$ and $y_2$.
\item Some pair of $x_1,x_2,y_1,y_2$ have two common neighbours. In this case the cop probes $a$ and the other two vertices.
\item Every pair of $x_1,x_2,y_1,y_2$ have a unique common neighbour. In this case the cop probes $a$, $x_1$ and $x_2$.
\end{enumerate}

Cases (i)--(v) exhaust all possibilities. In cases (i), (ii) and (iii), the probe at $x_1$ will distinguish between the two unprobed vertices in $\{x_1,x_2,y_1,y_2\}$, and so the cop wins immediately. In case (iv), the cop wins unless the robber was at one of the two unprobed vertices; these vertices have two common neighbours so she can now win using \observ{square}. In case (v), the six common neighbours of the pairs must all be different; write $z_{i,j}$ for the common neighbour of $x_i$ and $y_j$. The cop wins immediately unless the robber was at $y_1$ or $y_2$. In this case his possible locations at her next probe will be $\{w,y_1,y_2,z_{1,1},z_{1,2},z_{2,1},z_{2,2}\}$. The cop can therefore win by probing $x_1$, $y_1$ and $y_2$.\end{proof}

\section{Acknowledgements}
The first author acknowledges support from the European Union through funding under FP7--ICT--2011--8 project HIERATIC (316705) and from the European Research Council (ERC) under the European Union's Horizon 2020 research and innovation programme (grant agreement no.\ 639046), and is grateful to Douglas B. West for drawing his attention to this problem. The second author acknowledges support through funding from NSF grant DMS~1301614 and MULTIPLEX grant no.\ 317532, and is grateful to the organisers of the 8th Graduate Student Combinatorics Conference at the University of Illinois at Urbana-Champaign for drawing his attention to the problem. The third author acknowledges support through funding from the European Union under grant EP/J500380/1 as well as from the Studienstiftung des Deutschen Volkes.

\end{document}